\documentclass[a4paper,11pt]{article}
\title{A decomposition theorem for \{ISK4,wheel\}-free trigraphs}
\author{Martin Milani\v c\thanks{University of Primorska, UP IAM, UP FAMNIT, Koper, Slovenia.
Partially supported by the Slovenian Research Agency (I$0$-$0035$, research program P$1$-$0285$ and research projects N$1$-$0032$, J$1$-$5433$, J$1$-$6720$, J$1$-$6743$, and J$1$-$7051$). E-mail: \texttt{martin.milanic@upr.si}.
}
\and Irena Penev \thanks{Department of Applied Mathematics and Computer Science, Technical University of Denmark, Lyngby, Denmark. Most of this work was conducted while the author was at Universit\'e de Lyon, LIP, ENS de Lyon, Lyon, France. Partially
supported by the ANR Project \textsc{Stint} under \textsc{Contract ANR-13-BS02-0007}, by the LABEX MILYON (ANR-10-LABX-0070) of Universit\'e de
    Lyon, within the program ‘‘Investissements d'Avenir’’
    (ANR-11-IDEX-0007) operated by the French National Research Agency
    (ANR), and by the ERC Advanced Grant GRACOL, project number 320812. Email: \texttt{ipen@dtu.dk}.}
  \and Nicolas Trotignon\thanks{Unviersit\'e de Lyon, CNRS, LIP, ENS de
  Lyon. Partially supported by ANR project Stint under reference
  ANR-13-BS02-0007 and by the LABEX MILYON (ANR-10-LABX-0070) of
  Universit\'e de Lyon, within the program ‘‘Investissements
  d'Avenir’’ (ANR-11-IDEX-0007) operated by the French National
  Research Agency (ANR). E-mail: \texttt{nicolas.trotignon@ens-lyon.fr}.}
}

\usepackage{amsthm}
\usepackage{color}
\usepackage{amssymb}
\usepackage{amsbsy}
\usepackage{mathrsfs}
\usepackage{amstext}
\usepackage{stmaryrd}
\newtheorem{lemma}{Lemma}[section]
\newtheorem{proposition}[lemma]{Proposition}
\newtheorem{theorem}[lemma]{Theorem}

\newtheorem*{thm-decomp}{Theorem~\ref{thm-decomp}}
\begin{document}
\maketitle

\begin{abstract}
  An \emph{ISK4} in a graph $G$ is an induced subgraph of $G$ that is
  isomorphic to a subdivision of $K_4$ (the complete graph on four
  vertices). A \emph{wheel} is a graph that consists of a chordless cycle, together with a vertex that has at least three
  neighbors in the cycle.  A graph is \{ISK4,wheel\}-free if it has
  no ISK4 and does not contain a wheel as an induced subgraph. A ``trigraph''
  is a generalization of a graph in which some pairs of vertices have
  ``undetermined'' adjacency. We prove a decomposition theorem for
  \{ISK4,wheel\}-free trigraphs. Our proof closely follows the
  proof of a decomposition theorem for ISK4-free graphs due to
  L{\'e}v{\^e}que, Maffray, and Trotignon
  (On graphs with no induced subdivision of $K_4$. J.~Combin.~Theory Ser.~B, 102(4):924--947, 2012).
\end{abstract}

\section{Introduction}
\label{s:intro}

All graphs in this manuscript are finite and simple. If $H$ and $G$ are graphs, we say that $G$ is {\em $H$-free} if $G$ does not contain (an isomorphic copy of) $H$ as an induced subgraph. If $\mathcal{H}$ is a family of graphs, a graph $G$ is said to be {\em $\mathcal{H}$-free} if $G$ is $H$-free for all $H \in \mathcal{H}$.

An \emph{ISK4} in a graph $G$ is an induced subgraph of $G$ that is isomorphic to a subdivision of $K_4$ (the complete graph on four vertices). A \emph{wheel} is a graph that consists of a chordless cycle, together with a vertex that has at least three neighbors in the cycle. L{\'e}v{\^e}que, Maffray, and Trotignon~\cite{MR2927414} proved a decomposition theorem for ISK4-free graphs and then derived a decomposition theorem for \{ISK4,wheel\}-free graphs as a corollary. Here, we are interested in a class that generalizes the class of \{ISK4,wheel\}-free graphs, namely, the class of \{ISK4,wheel\}-free ``trigraphs.'' Trigraphs (originally introduced by Chudnovsky~\cite{chudnovsky:these,chudnovsky:trigraphs} in the context of Berge graphs) are a generalization of graphs in which certain pairs of vertices may have ``undetermined'' adjacency (one can think of such pairs as ``optional edges''). Every graph can be thought of as a trigraph: a graph is simply a trigraph with no ``optional edges.'' Trigraphs and related concepts are formally defined in Section~\ref{s:trigraphs}.

We now wish to state the decomposition theorem for \{ISK4,wheel\}-free graphs from~\cite{MR2927414}, but we first need a few definitions. A graph is {\em series-parallel} if it does not contain any subdivision of $K_4$ as a (not necessarily induced) subgraph.
The {\em line graph} of a graph $H$, denoted by $L(H)$, is the graph whose vertices are the edges of $H$, and in which two vertices (i.e., edges of $H$) are adjacent if they share an endpoint in $H$. A graph is {\em chordless} if all its cycles are induced.

If $H$ is an induced subgraph of a graph $G$ and $v \in V(G) \smallsetminus V(H)$, then the {\em attachment} of $v$ over $H$ in $G$ is the set of all neighbors of $v$ in $V(H)$. If $S$ is either a set of vertices or an induced subgraph of $G \smallsetminus V(H)$, then the {\em attachment} of $S$ over $H$ in $G$ is the set of all vertices of $H$ that are adjacent to at least one vertex of $S$. A {\em square} is a cycle of length four. If $S$ is an induced square of a graph $G$, say with vertices $a_1,a_2,a_3,a_4$ (with subscripts understood to be modulo $4$) that appear in that order in $S$, then a {\em long link} of $S$ in $G$ is an induced path $P$ of $G \smallsetminus V(S)$ that contains at least one edge, and satisfies the property that there is an index $i \in \{1,2,3,4\}$ such that the attachment of one endpoint of $P$ over $S$ is $\{a_i,a_{i+1}\}$, the attachment of the other endpoint of $P$ over $S$ is $\{a_{i+2},a_{i+3}\}$, and no interior vertex of $P$ has a neighbor in $S$. A {\em long rich square} is a graph $G$ that contains an induced square $S$ (called a {\em central square} of $G$) such that $G \smallsetminus V(S)$ contains at least two components, and all such components are long links of $S$ in $G$.

A {\em clique} of a graph $G$ is a (possibly empty) set of pairwise adjacent vertices of $G$, and a {\em stable set} of $G$ is a (possibly empty) set of pairwise non-adjacent vertices of $G$. A {\em cutset} of a graph $G$ is a (possibly empty) set of vertices whose deletion from $G$ yields a disconnected graph. A {\em clique-cutset} of a graph $G$ is clique of $G$ that is also a cutset of $G$. (Note that if $G$ is a disconnected graph, then $\emptyset$ is a clique-cutset of $G$.) A {\em proper 2-cutset} of a graph $G$ is a cutset $\{a,b\}$ of $G$ that is a stable set of size two such that $V(G) \smallsetminus \{a,b\}$ can be partitioned into two non-empty sets $X$ and $Y$ so that there is no edge between $X$ and $Y$ and neither $G[X \cup \{a,b\}]$ nor $G[Y \cup \{a,b\}]$ is a path between $a$ and $b$.

We are now ready to state the decomposition theorem for \{ISK4,wheel\}-free graphs from~\cite{MR2927414} (this is Theorem 1.2 from~\cite{MR2927414}).

\begin{theorem} \label{decomp-ISK4WheelFree-graph} \cite{MR2927414} Let $G$ be an \{ISK4,wheel\}-free graph. Then at least one of the following holds:
\begin{itemize}
\item $G$ is series-parallel;
\item $G$ is the line graph of a chordless graph with maximum degree at most three;
\item $G$ is a complete bipartite graph;
\item $G$ is a long rich square;
\item $G$ has a clique-cutset or a proper 2-cutset.
\end{itemize}
\end{theorem}

We remark, however, that the fourth outcome of Theorem~\ref{decomp-ISK4WheelFree-graph} (that is, the outcome that $G$ is a long rich square) is in fact unnecessary. To see this, suppose that $G$ is a long rich square, and let $S$ be a central square of $G$. Choose any two-edge path $S'$ of the square $S$, and choose two components, call them $P_1$ and $P_2$, of $G \smallsetminus V(S)$. By the definition of a long rich square, $P_1$ and $P_2$ are long links of $S$ in $G$, and we see that $W = G[V(S') \cup V(P_1) \cup V(P_2)]$ is a wheel. Indeed, if $x$ is the ``central'' vertex of $S'$ (i.e., the unique vertex of the two-edge path $S'$ that is adjacent to the other two vertices of $S'$), then $W \smallsetminus x$ is a chordless cycle, and $x$ has four neighbors in this cycle. Thus, long rich squares are not wheel-free (and consequently, they are not \{ISK4,wheel\}-free). This observation allows us to strengthen Theorem~\ref{decomp-ISK4WheelFree-graph} as follows.

\begin{theorem} \label{decomp-ISK4WheelFree-graph-strong} Let $G$ be an \{ISK4,wheel\}-free graph. Then at least one of the following holds:
\begin{itemize}
\item $G$ is series-parallel;
\item $G$ is the line graph of a chordless graph with maximum degree at most three;
\item $G$ is a complete bipartite graph;
\item $G$ has a clique-cutset or a proper 2-cutset.
\end{itemize}
\end{theorem}

Our goal in this manuscript is to prove a trigraph version of Theorem~\ref{decomp-ISK4WheelFree-graph-strong}. In Section~\ref{s:cyc}, we state a few lemmas about ``cyclically 3-connected graphs'' proven in~\cite{MR2927414}. In Section~\ref{s:trigraphs}, we define trigraphs and introduce some basic trigraph terminology. Finally, in Section~\ref{sec:dec}, we prove our decomposition theorem for \{ISK4,wheel\}-free trigraphs (Theorem~\ref{thm-decomp}), which is very similar to Theorem~\ref{decomp-ISK4WheelFree-graph-strong}. We prove Theorem~\ref{thm-decomp} by imitating the proof of the decomposition theorem for ISK4-free graphs from~\cite{MR2927414}. Interestingly, the fact that we work with trigraphs rather than graphs does not substantially complicate the proof. On the other hand, the fact that we restrict ourselves to the wheel-free case significantly simplifies the argument (indeed, some of the most difficult parts of the proof of the theorem for ISK4-free graphs from~\cite{MR2927414} involve ISK4-free graphs that contain induced wheels).

\section{Cyclically 3-connected graphs}
\label{s:cyc}

In this section, we state a few lemmas proven in~\cite{MR2927414}, but first, we need some definitions. Given a graph $G$, a vertex $u \in V(G)$, and a set $X \subseteq V(G) \smallsetminus \{u\}$, we say that $u$ is {\em complete} (respectively: {\em anti-complete}) to $X$ in $G$ provided that $u$ is adjacent (respectively: non-adjacent) to every vertex of $X$ in $G$. Given a graph $G$ and disjoint sets $X,Y \subseteq V(G)$, we say that $X$ is {\em complete} (respectively: {\em anti-complete}) to $Y$ in $G$ provided that every vertex of $X$ is complete (respectively: anti-complete) to $Y$ in $G$. A {\em separation} of a graph $H$ is a pair $(A,B)$ of subsets of $V(H)$ such that $V(H) = A \cup B$, and $A \smallsetminus B$ is anti-complete to $B \smallsetminus A$. A separation $(A,B)$ of $H$ is {\em proper} if both $A \smallsetminus B$ and $B \smallsetminus A$ are non-empty. A {\em $k$-separation} of $H$ is a separation $(A,B)$ of $H$ such that $|A \cap B| \leq k$. A separation $(A,B)$ is {\em cyclic} if both $H[A]$ and $H[B]$ have cycles. A graph $H$ is {\em cyclically 3-connected} if it is 2-connected, is not a cycle, and admits no cyclic 2-separation. Note that a cyclic 2-separation of any graph is proper. A {\em theta} is any subdivision of the complete bipartite graph $K_{2,3}$. As usual, if $H_1$ and $H_2$ are graphs, we denote by $H_1 \cup H_2$ the graph whose vertex set is $V(H_1) \cup V(H_2)$ and whose edge set is $E(H_1) \cup E(H_2)$.

The {\em length} of a path is the number of edges that it contains. A {\em branch vertex} in a graph $G$ is a vertex of degree at least three. A {\em branch} in a graph $G$ is an induced path $P$ of length at least one whose endpoints are branch vertices of $G$ and all of whose interior vertices are of degree two in $G$.

We now state the lemmas from~\cite{MR2927414} that we need. The five lemmas below are Lemmas 4.3, 4.5, 4.6, 4.7, and 4.8 from~\cite{MR2927414}, respectively. 

\begin{lemma}\label{lemma43}\cite{MR2927414} Let $H$ be a cyclically 3-connected graph, let $a$ and $b$ be two branch vertices of $H$, and let $P_1$, $P_2$, and $P_3$ be three induced paths of $H$ whose ends are $a$ and $b$. Then one of the following holds:
\begin{itemize}
\item $P_1$, $P_2$, $P_3$ are branches of $H$ of length at least two and $H = P_1 \cup P_2 \cup P_3$ (so $H$ is a theta);
\item there exist distinct indices $i,j \in \{1,2,3\}$ and a path $S$ of $H$ with one end in the interior of $P_i$ and the other end in the interior of $P_j$,  such that no interior vertex of $S$ belongs to $V(P_1 \cup P_2 \cup P_3)$, and such that $P_1 \cup P_2 \cup P_3 \cup S$ is a subdivision of $K_4$.
\end{itemize}
\end{lemma}

\begin{lemma}\label{lemma45}\cite{MR2927414} A graph is cyclically 3-connected if and only if it is either a theta or a subdivision of a 3-connected graph.
\end{lemma}

\begin{lemma}\label{lemma46}\cite{MR2927414} Let $H$ be a cyclically 3-connected graph, and let $a$ and $b$ be two distinct vertices of $H$. If no branch of $H$ contains both $a$ and $b$, then $H' = (V(H), E(H) \cup \{ab\})$ is a cyclically 3-connected graph and every graph obtained from $H'$ by subdividing ab is cyclically 3-connected.
\end{lemma}

\begin{lemma}\label{lemma47}\cite{MR2927414} Let $H$ be a cyclically 3-connected graph, let $Z$ be a cycle of $H$, and let $a$, $b$, $c$, $d$ be four pairwise distinct vertices of $Z$ that lie in this order on $Z$ and satisfy $ab,cd \in E(Z)$. Let $P$ be the subpath of $Z$ from $a$ to $d$ that does not contain $b$ and $c$, and let $Q$ be the subpath of $Z$ from $b$ to $c$ that does not contain $a$ and $d$. Suppose that there exist distinct branches $F_{ab}$ and $F_{cd}$ of $H$ such that $ab \in E(F_{ab})$ and $cd \in E(F_{cd})$. Then there is a path $R$ of $H$ such that one endpoint of $R$ belongs to $P$ and the other to $Q$, such that no interior vertex of $R$ belongs to $Z$, and such that $R$ is not from $a$ to $b$ or from $c$ to $d$.
\end{lemma}

\begin{lemma}\label{lemma48}\cite{MR2927414} Let $H$ be a subdivision of a 3-connected graph. Let $C$ be a cycle of $H$ and $e$ an edge of $H$ such that $C$ and $e$ are edgewise disjoint. Then some subgraph of $H$ that contains $C$ and $e$ is a subdivision of $K_4$.
\end{lemma}

\section{Trigraphs}
\label{s:trigraphs}

Given a set $S$, we denote by ${S \choose 2}$ the set of all subsets of $S$ of size two. A {\em trigraph} is an ordered pair $G = (V(G),\theta_G)$, where $V(G)$ is a finite set, called the {\em vertex set} of $G$ (members of $V(G)$ are called {\em vertices} of $G$), and $\theta_G:{V(G) \choose 2} \rightarrow \{-1,0,1\}$ is a function, called the {\em adjacency function} of $G$. The {\em null} trigraph is the trigraph whose vertex set is empty; a {\em non-null} trigraph is any trigraph whose vertex set is non-empty. If $G$ is a trigraph and $u,v \in V(G)$ are distinct, we usually write $uv$ instead of $\{u,v\}$ (note that this means that $uv = vu$), and furthermore:
\begin{itemize}
\item if $\theta_G(uv) = 1$, we say that $uv$ is a {\em strongly adjacent pair} of $G$, or that $u$ and $v$ are {\em strongly adjacent} in $G$, or that $u$ is {\em strongly adjacent} to $v$ in $G$, or that $v$ is a {\em strong neighbor} of $u$ in $G$, or that $u$ and $v$ are the {\em endpoints of a strongly adjacent pair} of $G$;
\item if $\theta_G(uv) = 0$, we say that $uv$ is a {\em semi-adjacent pair} of $G$, or that $u$ and $v$ are {\em semi-adjacent} in $G$, or that $u$ is {\em semi-adjacent} to $v$ in $G$, or that $v$ is a {\em weak neighbor} of $u$ in $G$, or that $u$ and $v$ are the {\em endpoints of a semi-adjacent pair} of $G$;
\item if $\theta_G(uv) = -1$, we say that $uv$ is a {\em strongly anti-adjacent pair} of $G$, or that $u$ and $v$ are {\em strongly anti-adjacent} in $G$, or that $u$ is {\em strongly anti-adjacent} to $v$ in $G$, or that $v$ is a {\em strong anti-neighbor} of $u$ in $G$, or that $u$ and $v$ are the {\em endpoints of a strongly anti-adjacent pair} of $G$;
\item if $\theta_G(uv) \geq 0$, we say that $uv$ is an {\em adjacent pair} of $G$, or that $u$ and $v$ are {\em adjacent} in $G$, or that $u$ is {\em adjacent} to $v$ in $G$, or that $v$ is a {\em neighbor} of $u$ in $G$, or that $u$ and $v$ are the {\em endpoints of an adjacent pair} of $G$;
\item if $\theta_G(uv) \leq 0$, we say that $uv$ is an {\em anti-adjacent pair} of $G$, or that $u$ and $v$ are {\em anti-adjacent} in $G$, or that $u$ is {\em anti-adjacent} to $v$ in $G$, or that $v$ is an {\em anti-neighbor} of $u$ in $G$, or that $u$ and $v$ are the {\em endpoints of an anti-adjacent pair} of $G$.
\end{itemize}
\noindent
Note that a semi-adjacent pair is simultaneously an adjacent pair and an anti-adjacent pair. One can think of strongly adjacent pairs as ``edges,'' of strongly anti-adjacent pairs as ``non-edges,'' and of semi-adjacent pairs as ``optional edges.'' Clearly, any graph can be thought of as a trigraph: a graph is simply a trigraph with no semi-adjacent pairs, that is, the adjacency function of a graph $G$ is a mapping from ${V(G) \choose 2}$ to the set $\{-1,1\}$.

Given a trigraph $G$, a vertex $u \in V(G)$, and a set $X \subseteq V(G) \smallsetminus \{u\}$, we say that $u$ is {\em complete} (respectively: {\em strongly complete, anti-complete, strongly anti-complete}) to $X$ in $G$ provided that $u$ is adjacent (respectively: strongly adjacent, anti-adjacent, strongly anti-adjacent) to every vertex of $X$ in $G$. Given a trigraph $G$ and disjoint sets $X,Y \subseteq V(G)$, we say that $X$ is {\em complete} (respectively: {\em strongly complete, anti-complete, strongly anti-complete)} to $Y$ in $G$ provided that every vertex of $X$ is complete (respectively: strongly complete, anti-complete, strongly anti-complete) to $Y$ in $G$.

Isomorphism between trigraphs is defined in the natural way. The {\em complement} of a trigraph $G = (V(G),\theta_G)$ is the trigraph $\overline{G} = (V(\overline{G}),\theta_{\overline{G}})$ such that $V(\overline{G}) = V(G)$ and $\theta_{\overline{G}} = -\theta_G$. Thus, $\overline{G}$ is obtained from $G$ by turning all strongly adjacent pairs of $G$ into strongly anti-adjacent pairs, and turning all strongly anti-adjacent pairs of $G$ into strongly adjacent pairs; semi-adjacent pairs of $G$ remain semi-adjacent in $\overline{G}$.

Given trigraphs $G$ and $\widetilde{G}$, we say that $\widetilde{G}$ is a {\em semi-realization} of $G$ provided that $V(\widetilde{G}) = V(G)$ and
for all distinct $u,v \in V(\widetilde{G}) = V(G)$, we have that if $\theta_G(uv) = 1$ then $\theta_{\widetilde{G}}(uv) = 1$, and if $\theta_G(uv) = -1$ then $\theta_{\widetilde{G}}(uv) = -1$. Thus, a semi-realization of a trigraph $G$ is any trigraph that can be obtained from $G$ by ``deciding'' the adjacency of some semi-adjacent pairs of $G$, that is, by possibly turning some semi-adjacent pairs of $G$ into strongly adjacent or strongly anti-adjacent pairs. (In particular, every trigraph is a semi-realization of itself.) A {\em realization} of a trigraph $G$ is a graph that is a semi-realization of $G$. Thus, a realization of a trigraph $G$ is any graph that can be obtained by ``deciding'' the adjacency of all semi-adjacent pairs of $G$, that is, by turning each semi-adjacent pair of $G$ into an edge or a non-edge. Clearly, if a trigraph $G$ has $m$ semi-adjacent pairs, then $G$ has $3^m$ semi-realizations and $2^m$ realizations. The {\em full realization} of a trigraph $G$ is the graph obtained from $G$ by turning all semi-adjacent pairs of $G$ into strongly adjacent pairs (i.e., edges), and the {\em null realization} of $G$ is the graph obtained from $G$ by turning all semi-adjacent pairs of $G$ into strongly anti-adjacent pairs (i.e., non-edges).

A {\em clique} (respectively: {\em strong clique}, {\em stable set}, {\em strongly stable set}) of a trigraph $G$ is a set of pairwise adjacent (respectively: strongly adjacent, anti-adjacent, strongly anti-adjacent) vertices of $G$. Note that any subset of $V(G)$ of size at most one is both a strong clique and a strongly stable set of $G$. Note also that if $S \subseteq V(G)$, then $S$ is a (strong) clique of $G$ if and only if $S$ is a (strongly) stable set of $\overline{G}$. Note furthermore that if $K$ is a strong clique and $S$ is a stable set of $G$, then $|K \cap S| \leq 1$; similarly, if $K$ is a clique and $S$ is a strongly stable set of $G$, then $|K \cap S| \leq 1$. However, if $K$ is a clique and $S$ is a stable set of $G$, then we are only guaranteed that vertices in $K \cap S$ are pairwise semi-adjacent to each other, and it is possible that $|K \cap S| \geq 2$. A {\em triangle} (respectively: {\em strong triangle}) is a clique (respectively: strong clique) of size three.

Given a trigraph $G$ and a set $X \subseteq V(G)$, the {\em subtrigraph of $G$ induced by $X$}, denoted by $G[X]$, is the trigraph with vertex set $X$ and adjacency function $\theta_G \upharpoonright {X \choose 2}$, where for a function $f:A\to B$ and a set $A'\subseteq A$, we denote by $f\upharpoonright  A'$ the restriction of $f$ to $A'$. Given $v_1,\dots,v_p \in V(G)$, we often write $G[v_1,\dots,v_p]$ instead of $G[\{v_1,\dots,v_p\}]$. If $H = G[X]$ for some $X \subseteq V(G)$, we also say that $H$ is an {\em induced subtrigraph} of $G$; when convenient, we relax this definition and say that $H$ is an induced subtrigraph of $G$ provided that there is some set $X \subseteq V(G)$ such that $H$ is isomorphic to $G[X]$. Further, for a trigraph $G$ and a set $X \subseteq V(G)$, we set $G \smallsetminus X = G[V(G) \smallsetminus X]$; for $v \in V(G)$, we often write $G \smallsetminus v$ instead of $G \smallsetminus \{v\}$. The trigraph $G \smallsetminus X$ (respectively: $G \smallsetminus v$) is called the subtrigraph of $G$ obtained by {\em deleting} $X$ (respectively: by {\em deleting} $v$).

If $H$ is a graph, we say that a trigraph $G$ is an {\em $H$-trigraph} if some realization of $G$ is (isomorphic to) $H$. Further, if $H$ is a graph and $G$ a trigraph, we say that $G$ is {\em $H$-free} provided that all realizations of $G$ are $H$-free (equivalently: provided that no induced subtrigraph of $G$ is an $H$-trigraph). If $\mathcal{H}$ is a family of graphs, we say that a trigraph $G$ is {\em $\mathcal{H}$-free} provided that $G$ is $H$-free for all graphs $H \in \mathcal{H}$. In particular, a trigraph is {\em ISK4-free} (respectively: {\em wheel-free}, {\em \{ISK4,wheel\}-free}) if all its realizations are ISK4-free (respectively: wheel-free, \{ISK4,wheel\}-free).

\begin{sloppypar}
Given a graph $H$, a trigraph $G$, vertices $v_1,\dots,v_p \in V(G)$, sets $X_1,\dots,X_q \subseteq V(G)$, and induced subtrigraphs $G_1,\dots,G_r$ of $G$ (with $p,q,r \geq 0$), we say that $v_1,\dots,v_p,X_1,\dots,X_q,G_1,\dots,G_r$ {\em induce an $H$-trigraph in $G$} provided that $G[\{v_1,\dots,v_p\} \cup X_1 \cup \dots \cup X_q \cup V(G_1) \cup \dots V(G_r)]$ is an $H$-trigraph.
\end{sloppypar}

A trigraph is {\em connected} if its full realization is a connected graph. A trigraph is {\em disconnected} if it is not connected. A {\em component} of a non-null trigraph $G$ is any (inclusion-wise) vertex-maximal connected induced subtrigraph of $G$. Clearly, if $H$ is an induced subtrigraph of a non-null trigraph $G$, then we have that $H$ is a component of $G$ if and only if the full realization of $H$ is a component of the full realization of $G$.

A trigraph is a {\em path} if at least one of its realizations is a path. A trigraph is a {\em narrow path} if its full realization is a path. We often denote a path $P$ by $v_0-v_1-\dots-v_k$ (with $k \geq 0$), where $v_0,v_1,\dots,v_k$ are the vertices of $P$ that appear in that order on some realization $\widetilde{P}$ of $P$ such that $\widetilde{P}$ is a path. The {\em endpoints} of a narrow path are the endpoints of its full realization; if $a$ and $b$ are the endpoints of a narrow path $P$, then we also say that $P$ is a {\em narrow $(a,b)$-path}, and that $P$ is a narrow path {\em between} $a$ and $b$. If $P$ is a narrow path and $a,b \in V(P)$, we denote by $a-P-b$ the minimal connected induced subtrigraph of $P$ that contains both $a$ and $b$ (clearly, $a-P-b$ is a narrow path between $a$ and $b$). The {\em interior vertices} of a narrow path are the interior vertices of its full realization. The {\em interior} of a narrow path is the set of all interior vertices of that narrow path. The {\em length} of a narrow path is one less than the number of vertices that it contains. (In other words, the length of a narrow path is the number of edges that its full realization has.) A {\em path} (respectively: {\em narrow path}) in a trigraph $G$ is an induced subtrigraph $P$ of $G$ such that $P$ is a path (respectively: narrow path).

Note that if $G$ is a connected trigraph, then for all vertices $a,b \in V(G)$, there exists a narrow path between $a$ and $b$ in $G$. (To see this, consider the full realization $\widetilde{G}$ of $G$. $\widetilde{G}$ is connected, and so there is a path in $\widetilde{G}$ between $a$ and $b$; let $P$ be a shortest such path in $\widetilde{G}$. The minimality of $P$ guarantees that $P$ is an induced path of $\widetilde{G}$. But now $G[V(P)]$ is a narrow path of $G$ between $a$ and $b$.)

A {\em hole} of a trigraph $G$ is an induced subtrigraph $C$ of $G$ such that some realization of $C$ is a chordless cycle of length at least four. We often denote a hole $C$ of $G$ by $v_0-v_1-\dots-v_{k-1}-v_0$ (with $k \geq 4$ and indices in $\mathbb{Z}_k$), where $v_0,v_1,\dots,v_{k-1}$ are the vertices of $C$ that appear in that order in some realization $\widetilde{C}$ of $C$ such that $\widetilde{C}$ is a chordless cycle of length at least four.

The {\em degree} of a vertex $v$ in a trigraph $G$, denoted by ${\rm deg}_G(v)$, is the number of neighbors that $v$ has in $G$. A {\em branch vertex} in a trigraph $G$ is a vertex of degree at least three. A {\em branch} in a trigraph $G$ is a narrow path $P$ between two distinct branch vertices of $G$ such that no interior vertex of $P$ is a branch vertex. A {\em flat branch} in a trigraph $G$ is a branch $P$ of $G$ such that no adjacent pair of $P$ is contained in a triangle of $G$. (Note that every branch of length at least two is flat.) If $a$ and $b$ are the endpoints of a branch (respectively: flat branch) $P$ of $G$, then we also say that $P$ is an {\em $(a,b)$-branch} (respectively: {\em $(a,b)$-flat branch}) of $G$.

A {\em cutset} of a trigraph $G$ is a (possibly empty) set $C \subseteq V(G)$ such that $G \smallsetminus C$ is disconnected. A {\em cut-partition} of a trigraph $G$ is a partition $(A,B,C)$ of $V(G)$ such that $A$ and $B$ are non-empty ($C$ may possibly be empty), and $A$ is strongly anti-complete to $B$. Note that if $(A,B,C)$ is a cut-partition of $G$, then $C$ is a cutset of $G$. Conversely, every cutset of $G$ induces at least one cut-partition of $G$. A {\em clique-cutset} of a trigraph $G$ is a (possibly empty) strong clique $C$ of $G$ such that $G \smallsetminus C$ is disconnected. A {\em cut-vertex} of a trigraph $G$ is a vertex $v \in V(G)$ such that $G \smallsetminus v$ is disconnected. Note that if $v$ is a cut-vertex of $G$, then $\{v\}$ is a clique-cutset of $G$. A {\em stable 2-cutset} of a trigraph $G$ is cutset of $G$ that is a stable set of size two. We remark that if $C$ is a cutset of a trigraph $G$ such that $|C| \leq 2$, then $C$ is either a clique-cutset or a stable 2-cutset of $G$.

A trigraph is {\em series-parallel} if its full realization is series-parallel (equivalently: if all its realizations are series-parallel).

A {\em complete bipartite trigraph} is a trigraph whose vertex set can be partitioned into two strongly stable sets that are strongly complete to each other. A {\em bipartition} of a complete bipartite trigraph $G$ is a partition $(A,B)$ of $V(G)$ such that $A$ and $B$ are strongly stable sets, strongly complete to each other. A complete bipartite trigraph $G$ is {\em thick} if both sets of its bipartition contain at least three vertices. A trigraph is a {\em strong $K_{3,3}$} if its full realization is a $K_{3,3}$ and it contains no semi-adjacent pairs. Clearly, a strong $K_{3,3}$ is a thick complete bipartite trigraph.

A {\em line trigraph} of a graph $H$ is a trigraph $G$ whose full realization is (isomorphic to) the line graph of $H$, and all of whose triangles are strong. A trigraph $G$ is said to be a {\em line trigraph} provided there is a graph $H$ such that $G$ is a line trigraph of $H$.

\section{A decomposition theorem for \{ISK4,wheel\}-free trigraphs}
\label{sec:dec}

The main result of this manuscript is the following decomposition theorem for \{ISK4,wheel\}-free graphs.

\begin{theorem} \label{thm-decomp} Let $G$ be an \{ISK4,wheel\}-free trigraph. Then $G$ satisfies at least one of the following:
\begin{itemize}
\item $G$ is a series-parallel trigraph;
\item $G$ is a complete bipartite trigraph;
\item $G$ is a line trigraph;
\item $G$ admits a clique-cutset;
\item $G$ admits a stable 2-cutset. 
\end{itemize}
\end{theorem}

The remainder of this manuscript is devoted to proving Theorem~\ref{thm-decomp}. We begin with a simple but useful proposition.

\begin{proposition} \label{delete-three} Let $G$ be a non-null connected trigraph that is not a narrow path. Then $G$ contains three distinct vertices such that the deletion of any one of them yields a connected trigraph.
\end{proposition}
\begin{proof}
Let $\widetilde{G}$ be the full realization of $G$. It suffices to show that $\widetilde{G}$ contains three distinct vertices such that the deletion of any one of them from $\widetilde{G}$ yields a connected graph. Since $G$ is connected and not a narrow path, we know that $\widetilde{G}$ is connected and not a path. If $\widetilde{G}$ is a cycle, then the deletion of any one of its vertices yields a connected graph, and since any cycle has at least three vertices, we are done. So assume that $\widetilde{G}$ is not a cycle. Then $\widetilde{G}$ contains a vertex $x$ of degree at least three. Let $T$ be a breadth-first search spanning tree of $\widetilde{G}$ rooted at $x$. Then ${\rm deg}_T(x) = {\rm deg}_{\widetilde{G}}(x) \geq 3$. Since $T$ is a tree that contains a vertex of degree at least three, we know that $T$ contains at least three leaves. But clearly, for any leaf $v$ of $T$, the graph $\widetilde{G} \smallsetminus v$ (and hence also the trigraph $G \smallsetminus v$) is connected.
\end{proof}

\subsection{Diamonds in wheel-free trigraphs}

The {\em diamond} is the graph obtained by deleting one edge from $K_4$. Equivalently, the {\em diamond} is the unique (up to isomorphism) graph on four vertices and five edges.

\begin{proposition} \label{prop-diamond} Let $G$ be a $\{K_4,{\rm wheel}\}$-free trigraph. Then either $G$ is diamond-free, or $G$ admits a clique-cutset, or $G$ admits a stable 2-cutset.
\end{proposition}
\begin{proof}
We assume that $G$ is not diamond-free, and that it contains no cut-vertices, for otherwise we are done. Using the fact that $G$ is not diamond-free, we fix an inclusion-wise maximal clique $C$ of size at least two such that at least two vertices in $V(G) \smallsetminus C$ are complete to $C$. Let $A$ be the set of all vertices in $V(G) \smallsetminus C$ that are complete to $C$ in $G$. By construction, $|C|,|A| \geq 2$, and since $G$ is $K_4$-free, we deduce that $|C| = 2$ (set $C = \{c_1,c_2\}$), and that $A$ is a strongly stable set. Now, we claim that $C$ is a cutset of $G$. Suppose otherwise. Then there exists a narrow path in $G \smallsetminus C$ between two distinct vertices in $A$; among all such narrow paths, choose a narrow path $P$ of minimum length, and let $a,a' \in A$ be the two endpoints of $P$. By the minimality of $P$, we know that $V(P) \cap A = \{a,a'\}$. Furthermore, since $A$ is a strongly stable set, we know that $P$ is of length at least two.

If at least one vertex of $C$, say $c_1$, is anti-complete to $V(P) \smallsetminus \{a,a'\}$, then $a-P-a'-c_1-a$ is a hole in $G$, and $c_2$ has at least three neighbors (namely, $a$, $a'$, and $c_1$) in it, contrary to the fact that $G$ is wheel-free. Thus, neither $c_1$ nor $c_2$ is anti-complete to $V(P) \smallsetminus \{a,a'\}$. Further, since $V(P) \cap A = \{a,a'\}$, we know that no interior vertex of $P$ is adjacent to both $c_1$ and $c_2$. Now, let $p_1$ be the (unique) interior vertex of $P$ such that $p_1$ is has a neighbor in $C$, and such that the interior of the narrow path $a-P-p_1$ is strongly anti-complete to $C$. By symmetry, we may assume that $p_1$ is adjacent to $c_1$ (and therefore strongly anti-adjacent to $c_2$). Next, let $p_2$ be the (unique) interior vertex of $P$ such that $p_2$ is adjacent to $c_2$, and $c_2$ is strongly anti-complete to the interior of $a-P-p_2$. Note that $p_1$ is an interior vertex of the narrow path $a-P-p_2$. Now $a-P-p_2-c_2-a$ is a hole in $G$, and $c_1$ has at least three neighbors (namely, $a$, $p_1$, and $c_2$) in it, contrary to the fact that $G$ is wheel-free. This proves that $C$ is a cutset of $G$. Since $C$ is a cutset of size two of $G$, we see that $C$ is either a clique-cutset or a stable 2-cutset of $G$.
\end{proof}

\subsection{Attachment to a prism}

A {\em prism} is a trigraph $K$ that consists of two vertex-disjoint strong triangles, call them $\{x,y,z\}$ and $\{x',y',z'\}$, an $(x,x')$-flat branch $P_x$, a $(y,y')$-flat branch $P_y$, and a $(z,z')$-flat branch $P_z$, such that $xy$ and $x'y'$ are the only adjacent pairs between $P_x$ and $P_y$, and similarly for the other two pairs of flat branches. Sets $\{x,y,z\}$, $\{x',y',z'\}$, $V(P_x)$, $V(P_y)$, and $V(P_z)$ are called the {\em pieces} of the prism. We call $\{x,y,z\}$ and $\{x',y',z'\}$ the {\em triangle pieces} of the prism, and we call $V(P_x)$, $V(P_y)$, and $V(P_z)$ the {\em branch pieces} of the prism. When convenient, we refer to the flat branches $P_x$, $P_y$, and $P_z$ of $K$ (rather than to the sets $V(P_x)$, $V(P_y)$, and $V(P_z)$) as the branch pieces of $K$. 

In what follows, we will often consider a prism $K$ (with $P_x$, $P_y$, and $P_z$ as in the definition of a prism) that is an induced subtrigraph of a trigraph $G$. We remark that in this case, $P_x$, $P_y$, and $P_z$ need only be flat branches of $K$, and not necessarily of $G$.

Let $G$ be a trigraph, and let $H$ and $C$ be induced subtrigraphs of $G$ on disjoint vertex sets. The {\em attachment} of $C$ over $H$ in $G$ is the set of all vertices of $H$ that have a neighbor in $C$. Furthermore,
\begin{itemize}
\item $C$ is of {\em type triangle} with respect to $H$ in $G$ provided that the attachment of $C$ over $H$ is contained in a strong triangle of $H$;
\item $C$ is of {\em type branch} with respect to $H$ in $G$ provided that the attachment of $C$ over $H$ is contained in a flat branch of $H$.
\item $C$ is an {\em augmenting path} of $H$ in $G$ provided all the following are satisfied:
\begin{itemize}
\item $C$ is a narrow path of length at least one;
\item the interior of $C$ is strongly anti-complete to $V(H)$ in $G$;
\item if $a$ and $b$ are the endpoints of $C$, then there exist two distinct flat branches of $H$, call them $F_a$ and $F_b$, such that for each $x \in \{a,b\}$, $x$ has exactly two neighbors (call them $x_1$ and $x_2$) in $H$, $\{x,x_1,x_2\}$ is a strong triangle of $G$, and $x_1,x_2 \in V(F_x)$ (note that this implies that $x_1x_2$ is a strongly adjacent pair of $F_x$).
\end{itemize}
\end{itemize}

If $G$ is a trigraph, $H$ an induced subtrigraph of $G$, and $v \in V(G) \smallsetminus V(H)$, then the {\em attachment} of $v$ over $H$ in $G$ is the set of all vertices of $H$ that are adjacent to $v$; furthermore, we say that $v$ is of {\em type triangle} (respectively: of {\em type branch}) with respect to $H$ in $G$ provided that $G[v]$ is of type triangle (respectively: of type branch) with respect to $H$ in $G$.

\begin{proposition} \label{prism-vertex} Let $G$ be an \{ISK4,wheel,diamond\}-free trigraph, let $K$ be an induced subtrigraph of $G$ such that $K$ is a prism, and let $v \in V(G) \smallsetminus V(K)$. Then $v$ has at most two neighbors in $K$, and furthermore, $v$ is of type branch with respect to $K$.
\end{proposition}
\begin{proof}
Let $\{x,y,z\}$, $\{x',y',z'\}$, $P_x$, $P_y$, and $P_z$ be the pieces of the prism $K$, as in the definition of a prism. Note that $v$ has at most one neighbor in $\{x,y,z\}$, for otherwise, $G[v,x,y,z]$ would be either a diamond-trigraph or a $K_4$-trigraph, contrary to the fact that $G$ is $\{{\rm ISK4},{\rm diamond}\}$-free. Similarly, $v$ has at most one neighbor in $\{x',y',z'\}$. Suppose first that $v$ has neighbors in each of $P_x$, $P_y$, and $P_z$. Let $x^L$ be the neighbor of $v$ in $P_x$ such that $v$ is strongly anti-complete to the interior of the narrow path $x-P_x-x^L$, and let $y^L$ and $z^L$ be chosen analogously. Then $v$, $x-P_x-x^L$, $y-P_y-y^L$, and $z-P_z-z^L$ induce an ISK4-trigraph in $G$ (here, we use the fact that $v$ has at most one neighbor in $\{x',y',z'\}$), contrary to the fact that $G$ is ISK4-free. Thus, $v$ has neighbors in at most two of $P_x$, $P_y$, and $P_z$, and by symmetry, we may assume that $v$ is strongly anti-complete to $P_z$. If $v$ has more than two neighbors in $K$ (and therefore in $V(P_x) \cup V(P_y)$, since $v$ is strongly anti-complete to $P_z$), then $v$, $P_x$, and $P_y$ induce a wheel-trigraph in $G$, which is a contradiction. Thus, $v$ has at most two neighbors in $K$.

It remains to show that $v$ is of type branch. We may assume that $v$ has a unique neighbor (call it $v_x$) in $P_x$, and a unique neighbor (call it $v_y$) in $P_y$, for otherwise, $v$ is of type branch with respect to $K$, and we are done. Since $v$ has at most one neighbor in $\{x,y,z\}$ and at most one neighbor in $\{x',y',z'\}$, we know that either $v_x \neq x'$ and $v_y \neq y$, or $v_x \neq x$ and $v_y \neq y'$; by symmetry, we may assume that $v_x \neq x'$ and $v_y \neq y$. But now $v$, $x-P_x-v_x$, $P_y$, and $P_z$ induce an ISK4-trigraph in $G$, contrary to the fact that $G$ is ISK4-free.
\end{proof}

\begin{lemma} \label{prism-conn} Let $G$ be an \{ISK4,wheel,diamond\}-free trigraph, let $K$ be an induced subtrigraph of $G$ such that $K$ is a prism, and let $P$ be an inclusion-wise minimal connected induced subtrigraph of $G \smallsetminus V(K)$ such that $P$ is neither of type branch nor of type triangle with respect to $K$. Then $P$ is an augmenting path of $K$.
\end{lemma}
\begin{proof}
Let $\{x,y,z\}$, $\{x',y',z'\}$, $P_x$, $P_y$, and $P_z$ be the pieces of the prism $K$, as in the definition of a prism. By Proposition~\ref{prism-vertex}, $P$ contains more than one vertex. Let us show that $P$ is a narrow path. Suppose otherwise; then Proposition~\ref{delete-three} guarantees that $P$ contains three distinct vertices, call them $a$, $b$, and $c$, such that the deletion of any one of them from $P$ yields a connected trigraph. By the minimality of $P$, each of $P \smallsetminus a$, $P \smallsetminus b$, and $P \smallsetminus c$ is of type branch or triangle with respect to $K$. Then for each $v \in \{a,b,c\}$, there exists a piece $X_v$ of the prism $K$ such that the attachment of $P \smallsetminus v$ over $K$ is contained in $X_v$, and $v$ has a neighbor $\widetilde{v}$ in $V(K) \smallsetminus X_v$. Then $\widetilde{a} \in (X_b \cap X_c) \smallsetminus X_a$, $\widetilde{b} \in (X_a \cap X_c) \smallsetminus X_b$, and $\widetilde{c} \in (X_a \cap X_b) \smallsetminus X_c$. Thus, $X_a$, $X_b$, and $X_c$ are pairwise distinct and pairwise intersect. But no three pieces of $K$ have this property. This proves that $P$ is a narrow path.

Let $p$ and $p'$ be the endpoints of the narrow path $P$; since $P$ has at least two vertices, we know that $p \neq p'$. By the minimality of $P$, we know that there exist distinct pieces $A$ and $A'$ of the prism $K$ such that the attachment of $P \smallsetminus p'$ over $K$ is included in $A$, and the attachment of $P \smallsetminus p$ over $K$ is included in $A'$. Thus, the attachment of $p$ over $K$ is included in $A$, the attachment of $p'$ over $K$ is included in $A'$, and the attachment of every interior vertex of $P$ over $K$ is included in $A \cap A'$.
Since $A$ and $A'$ are distinct pieces of $K$, we see that $|A \cap A'| \leq 1$. Furthermore, by the minimality of $P$, the attachment of $p$ over $K$ is non-empty, as is the attachment of $p'$ over $K$.

Let us show that the interior of the narrow path $P$ is strongly anti-complete to $K$. Suppose otherwise. Then $A \cap A' \neq \emptyset$, and it follows that one of $A$ and $A'$ is a triangle and the other a flat branch of $K$; by symmetry, we may assume that $A = \{x,y,z\}$ and $A' = V(P_x)$. Thus, every interior vertex of $P$ is either strongly anti-complete to $K$, or has exactly one neighbor (namely $x$) in $K$; furthermore, at least one interior vertex of $P$ is adjacent to $x$ (because the attachment of the interior of $P$ over $K$ is non-empty). By Proposition~\ref{prism-vertex}, $p$ is of type branch, and it therefore has at most one neighbor in the triangle $\{x,y,z\}$; since the attachment of $p$ over $K$ is included in $\{x,y,z\}$, it follows that $p$ has exactly one neighbor in $K$ (and that neighbor belongs to the set $\{x,y,z\}$). If $x$ is the unique neighbor of $p$ in $K$, then $P$ is of type branch with respect to $K$, which is a contradiction. So by symmetry, we may assume that $y$ is the unique neighbor of $p$ in $K$. Since $P$ is not of type triangle, we know that $p'$ has a neighbor in $V(P_x) \smallsetminus \{x\}$; let $v' \in V(P_x) \smallsetminus \{x\}$ be the neighbor of $p'$ such that $p'$ is strongly anti-complete to the interior of of $v'-P_x-x'$. Then $P$, $v'-P_x-x'$, $P_y$, and $P_z$ induce an ISK4-trigraph in $G$, which is a contradiction. This proves that the interior of $P$ is strongly anti-complete to $K$.

Now, by Proposition~\ref{prism-vertex}, we know that both $p$ and $p'$ are of type branch; since the interior of $P$ is strongly anti-complete to $K$, and since $P$ is not of type branch with respect to $K$, we may assume by symmetry that the attachment of $p$ over $K$ is included in $P_x$, while the attachment of $p'$ over $K$ is included in $P_y$. Recall that the attachment of $p$ over $K$ is non-empty, as is the attachment of $p'$ over $K$. Further, by Proposition~\ref{prism-vertex}, each of $p$ and $p'$ has at most two neighbors in $K$. If the attachment of $P$ over $K$ contains at least three vertices, and at most three vertices in the attachment of $P$ over $K$ have a strong neighbor in $\{p,p'\}$, then it is easy to see that $P_x$, $P_y$, and $P$ induce an ISK4-trigraph in $G$, contrary to the fact that $G$ is ISK4-free. Thus, either each of $p,p'$ has a unique neighbor in $K$, or each of $p,p'$ has exactly exactly two neighbors (both of them strong) in $K$.

Suppose first that each of $p$ and $p'$ has a unique neighbor in $K$. Let $w$ be the unique neighbor of $p$ in $K$, and let $w'$ be the unique neighbor of $p'$ in $K$; by construction, $w \in V(P_x)$ and $w' \in V(P_y)$. Since $P$ is not of type triangle with respect to $K$, we know that $\{w,w'\} \neq \{x,y\}$ and $\{w,w'\} \neq \{x',y'\}$. Clearly then, either $w \neq x$ and $w' \neq y'$, or $w \neq x'$ and $w' \neq y$; by symmetry, we may assume that $w \neq x'$ and $w' \neq y$. But then $P$, $x-P_x-w$, $P_y$, and $P_z$ induce an ISK4-trigraph in $G$, contrary to the fact that $G$ is ISK4-free.

We now have that each of $p$ and $p'$ has exactly two neighbors in $K$, and that each of those neighbors is strong. Let $w_1,w_2$ be the two neighbors of $p$ in $K$, and let $w_1',w_2'$ be the two neighbors of $p'$ in $K$. Clearly, $w_1,w_2 \in V(P_x)$ and $w_1',w_2' \in V(P_y)$; by symmetry, we may assume that $w_2$ does not lie on the narrow path $x-P_x-w_1$, and that $w_2'$ does not lie on the narrow path $y-P_y-w_1'$. Now, if $w_1w_2$ and $w_1'w_2'$ are strongly adjacent pairs, then $P$ is an augmenting path for $K$ in $G$, and we are done. So assume that at least one of $w_1w_2$ and $w_1'w_2'$ is an anti-adjacent pair; by symmetry, we may assume that $w_1w_2$ is an anti-adjacent pair. But now $x-P_x-w_1$, $w_2-P_x-x'$, $P_y$, and $P$ induce an ISK4-trigraph in $G$, which is a contradiction. This completes the argument.
\end{proof}

\subsection{Line trigraphs}

We remind the reader that a {\em line trigraph} of a graph $H$ is a trigraph $G$ whose full realization is the line graph of $H$, and all of whose triangles are strong. We also remind the reader that a graph is {\em chordless} if all its cycles are induced. A {\em chordless subdivision} of a graph $H$ is any chordless graph that can be obtained by possibly subdividing the edges of $H$. We observe that if $G$ is a wheel-free trigraph that is a line trigraph of a graph $H$, then $H$ is chordless. (In fact, it is not hard to see that if $G$ is a line trigraph of a graph $H$, then $G$ is wheel-free if and only if $H$ is chordless. However, we do not use this stronger fact in what follows.)

\begin{proposition} \label{K4-vertex} Let $G$ be an \{ISK4,wheel,diamond\}-free trigraph,
let $K$ be an induced subtrigraph of $G$ such that $K$ is a line trigraph of a chordless subdivision $H$ of $K_4$, and let $v \in V(G) \smallsetminus V(K)$. Then $v$ has at most two neighbors in $K$, and furthermore, $v$ is of type branch with respect to $K$.
\end{proposition}
\begin{proof}
Since $H$ is a chordless subdivision of $K_4$, we know that each edge of $K_4$ is subdivided at least once to obtain $H$. Let $a$, $b$, $c$, and $d$ be the four vertices of $H$ of degree three. For each $x \in \{a,b,c,d\}$, the three edges of $H$ incident with $x$ form a strong triangle of $K$, and we label this strong triangle $T_x$. In $K$, for all distinct $x,y \in \{a,b,c,d\}$, there is a unique narrow path (which we call $P_{xy}$) such that one endpoint of this narrow path belongs to $T_x$, the other endpoint belongs to $T_y$, and no interior vertex of this narrow path belongs to any one of the four strong triangles $T_a$, $T_b$, $T_c$, and $T_d$. For all distinct $x,y \in \{a,b,c,d\}$, we have that $P_{xy}$ is of length at least one (because $H$ is obtained by subdividing each edge of $K_4$ at least once) and that $P_{xy} = P_{yx}$; furthermore, the six narrow paths ($P_{ab}$, $P_{ac}$, $P_{ad}$, $P_{bc}$, $P_{bd}$, and $P_{cd}$) are vertex-disjoint. Finally, we have that $V(K) = V(P_{ab}) \cup V(P_{ac}) \cup V(P_{ad}) \cup V(P_{bc}) \cup V(P_{bd}) \cup V(P_{cd})$, and for all adjacent pairs $uu'$ of $K$, we have that either there exists some $x \in \{a,b,c,d\}$ such that $u,u' \in T_x$, or there exist distinct $x,y \in \{a,b,c,d\}$ such that $u,u' \in V(P_{xy})$.

For all distinct $x,y \in \{a,b,c,d\}$, set $K_{xy} = G \smallsetminus V(P_{xy})$, and note that $K_{xy}$ is a prism, and so by Proposition~\ref{prism-vertex}, $v$ has at most two neighbors in $K_{xy}$ and is of type branch with respect to $K_{xy}$. Therefore, for all distinct $x,y \in \{a,b,c,d\}$, $v$ is strongly anti-complete to some flat branch of $K_{xy}$. Consequently, there exist some distinct $x,y \in \{a,b,c,d\}$ such that $v$ is strongly anti-complete to $P_{xy}$; by symmetry, we may assume that $v$ is strongly anti-complete to $P_{ab}$. Since (by Proposition~\ref{prism-vertex}), $v$ has at most two neighbors in $K_{ab}$, it follows that $v$ has at most two neighbors in $K$. It remains to show that $v$ is of type branch with respect to $K$. If $v$ has at most one neighbor in $K_{ab}$ (and therefore in $K$), then this is immediate. So assume that $v$ has exactly two neighbors (call them $v_1$ and $v_2$) in $K_{ab}$ (and therefore in $K$). By Proposition~\ref{prism-vertex}, $v$ is of type branch with respect to $K_{ab}$. Consequently, we have that either $v_1,v_2 \in V(P_{ac}) \cup V(P_{ad})$, or $v_1,v_2 \in V(P_{bc}) \cup V(P_{bd})$, or $v_1,v_2 \in V(P_{cd})$. Let us assume that $v$ is not of type branch with respect to $K$. Then by symmetry, we may assume that $v_1 \in V(P_{bc})$ and $v_2 \in V(P_{bd})$. But now $v$ is not of type branch with respect to the prism $K_{cd}$, contrary to Proposition~\ref{prism-vertex}. This completes the argument.
\end{proof}

\begin{lemma} \label{K4-conn} Let $G$ be an \{ISK4,wheel,diamond\}-free trigraph, let $K$ be an induced subtrigraph of $G$ such that $K$ is a line trigraph of a chordless subdivision $H$ of $K_4$, and let $P$ be an inclusion-wise minimal connected induced subtrigraph of $G \smallsetminus V(K)$ such that $P$ is neither of type branch nor of type triangle with respect to $K$. Then $P$ is an augmenting path for $K$.
\end{lemma}
\begin{proof}
Let vertices $a,b,c,d$ of $H$, strong triangles $T_a,T_b,T_c,T_d$ of $K$, narrow paths $P_{ab},P_{ac},P_{ad},P_{bc},P_{bd},P_{cd}$ of $K$, and prisms $K_{ab},K_{ac},K_{ad},K_{bc},K_{bd},K_{cd}$ be as in the proof of Proposition~\ref{K4-vertex}. We call $T_a$, $T_b$, $T_c$, $T_d$, $V(P_{ab})$, $V(P_{ac})$, $V(P_{ad})$, $V(P_{bc})$, $V(P_{bd})$, $V(P_{cd})$ the {\em pieces} of $K$. Set $T_a = \{v_{a,b},v_{a,c},v_{a,d}\}$, $T_b = \{v_{b,a},v_{b,c},v_{b,d}\}$, $T_c = \{v_{c,a},v_{c,b},v_{c,d}\}$, and $T_d = \{v_{d,a},v_{d,b},v_{d,c}\}$, so that for all distinct $x,y \in \{a,b,c,d\}$, the endpoints of $P_{xy}$ are $v_{x,y}$ and $v_{y,x}$.

We first show that $P$ is a narrow path. If not, then by Proposition~\ref{delete-three}, there exist distinct vertices $v_1,v_2,v_3 \in V(P)$ such that for all $i \in \{1,2,3\}$, $P \smallsetminus v_i$ is connected. By the minimality of $P$, there exist pieces $X_1,X_2,X_3$ of $K$ such that for all $i \in \{1,2,3\}$, the attachment of $P \smallsetminus v_i$ over $K$ is included in $X_i$, and $v_i$ has a neighbor (call it $y_i$) in $K \smallsetminus X_i$. Then $y_1 \in (X_2 \cap X_3) \smallsetminus X_1$, $y_2 \in (X_1 \cap X_3) \smallsetminus X_2$, and $y_3 \in (X_1 \cap X_2) \smallsetminus X_3$. Thus, $X_1$, $X_2$, and $X_3$ are pairwise distinct, and they pairwise intersect. But this is impossible because one easily sees by inspection that no three pieces of $K$ have this property. Thus, $P$ is a narrow path. Furthermore, since $P$ is not of type branch with respect to $K$, Proposition~\ref{K4-vertex} guarantees that $P$ is of length at least one. Let $p$ and $p'$ be the endpoints of $P$.

By the minimality of $P$, we know that there exists a piece $A$ of $P$ such that the attachment of $P \smallsetminus p'$ is included in $A$, and there exists a piece $A'$ of $P$ such that the attachment of $P \smallsetminus p$ over $K$ is included in $A'$. Then the attachment of $P$ over $K$ is included in $A \cup A'$, and the attachment of the interior of $P$ over $K$ is included in $A \cap A'$. Further, since $P$ is neither of type branch nor of type triangle with respect to $K$, we know that $A \neq A'$.

Let us show that the interior of $P$ is strongly anti-complete to $K$. Since the attachment of the interior of $P$ over $K$ is included in $A \cap A'$, we may assume that $A \cap A' \neq \emptyset$. Thus, $A$ and $A'$ are distinct pieces of $K$ that have a non-empty intersection. Then there exist distinct $x,y \in \{a,b,c,d\}$ such that one of $A$ and $A'$ is $V(P_{xy})$ and the other one is $T_x$; by symmetry, we may assume that $A = V(P_{ab})$ and $A' = T_a$, so that the attachment of $P$ over $K$ is included in $V(P_{ab}) \cup T_a$; in particular then, $P$ is strongly anti-complete to $P_{cd}$, and the attachment of the interior of $P$ over $K$ is included in $A \cap A' = \{v_{ab}\}$. Now consider the prism $K_{cd}$. We know that the attachment of $P$ over $K_{cd}$ is the same as the attachment of $P$ over $K$, and it is easy to see that $P$ is a minimal connected induced subtrigraph of $G \smallsetminus V(K)$ that is neither of type branch nor of type triangle with respect to $K_{cd}$. Then by Lemma~\ref{prism-conn}, $P$ is an augmenting path of $K_{cd}$, and in particular, the interior of $P$ is strongly anti-complete to $K_{cd}$, and therefore (since $P$ is strongly anti-complete to $P_{cd}$), the interior of $P$ is strongly anti-complete to $K$, as we had claimed.

We now have that the attachment of $p$ over $K$ is included in $A$, the attachment of $p'$ over $K$ is included in $A'$, and the interior of $P$ is strongly anti-complete to $K$. By Proposition~\ref{K4-vertex}, we know that both $p$ and $p'$ are of type branch, and so we may assume that $A$ and $A'$ are both vertex sets of flat branches of $K$. By symmetry, we may assume that $A = V(P_{ab})$ and $A' \in \{V(P_{ac}),V(P_{cd})\}$. Since $P$ is not of type branch with respect to $K$, we know that each of $p$ and $p'$ has a neighbor in $K$. Furthermore, if $A' = V(P_{ac})$, then either $p$ has a neighbor in $V(P_{ab}) \smallsetminus \{v_{a,b}\}$ or $p'$ has a neighbor in $V(P_{ac}) \smallsetminus \{v_{a,c}\}$, for otherwise, the attachment of $P$ over $K$ is included in $T_a$, contrary to the fact that $P$ is not of type triangle with respect to $K$.

Now, if $A' = V(P_{ac})$, then set $K' = K_{cd}$, and if $A' = V(P_{cd})$, then set $K' = K_{ac}$. It is easy to see that $P$ is a minimal connected induced subtrigraph of $G \smallsetminus K'$ that is neither of type branch nor of type triangle with respect to $K'$. By Lemma~\ref{prism-conn}, $P$ is an augmenting path for $K'$. Since the attachment of $p$ over $K$ is included in a flat branch of $K$, as is the attachment of $p'$ over $K$, and since the interior of $P$ is strongly anti-complete to $K$, we deduce that $P$ is an augmenting path for $K$ in $G$. This completes the argument.
\end{proof}

\begin{proposition} \label{prop-dist-flat-branch} Let $H$ be a chordless subdivision of a 3-connected graph $H_0$, and let $E \subseteq E(H)$. Assume that the edges in $E$ do not all belong to the same flat branch of $H$, and that some two edges of $E$ are vertex-disjoint. Then $H$ contains an induced subgraph $K$ such that $K$ is a chordless subdivision of $K_4$, and such that there exist vertex-disjoint edges $ab,cd \in E \cap E(K)$ that do not belong to the same flat branch of $K$.
\end{proposition}
\begin{proof}
We first observe that each edge of $H_0$ was subdivided at least once to obtain $H$. For suppose that some edge $uv$ of $H_0$ remained unsubdivided in $H$. Since $H_0$ is 3-connected, Menger's theorem guarantees that there are three internally disjoint paths between $u$ and $v$ in $H_0$. At least two of those paths do not use the edge $uv$, and by putting them together, we obtain a cycle $Z_0$ of $H_0$; clearly, the edge $uv$ is a chord of $Z_0$ in $H_0$. Now, a subdivision $Z$ of $Z_0$ is a cycle of $H$, and since $uv$ remained unsubdivided in $H$, we see that $uv$ is a chord of the cycle $Z$ in $H$, contrary to the fact that $H$ is chordless. Thus, no edge of $H_0$ remained unsubdivided in $H$. Note that this implies that $H$ is triangle-free, and consequently, that all branches of $H$ are flat.

Next, note that the branch vertices of $H$ are precisely the vertices of $H_0$. Furthermore, the flat branches of $H$ are precisely the paths of $H$ that were obtained by subdividing the edges of $H_0$. This implies that every edge of $H$ belongs to a unique flat branch of $H$, and it also implies that no two distinct flat branches of $H$ share more than one vertex. Finally, the branch vertices of $H$ (that is, the vertices of $H_0$) all belong to more than one flat branch of $H$.

We now claim that there exist vertex-disjoint edges $ab,cd \in E$ such that $ab$ and $cd$ do not belong to the same flat branch of $H$. By hypothesis, there exist vertex-disjoint edges $ab,a'b' \in E$. If $ab$ and $a'b'$ do not belong to the same flat branch of $H$, then we set $c = a'$ and $d = b'$, and we are done. Suppose now that $ab$ and $a'b'$ do belong to the same flat branch (call it $F$) of $H$. Since edges in $E$ do not all belong to the same flat branch of $H$, there exists an edge $cd \in E$ such that $cd$ does not belong to the flat branch $F$. Since every edge of $H$ belongs to a unique flat branch of $H$, and since two distinct flat branches of $H$ have at most one vertex in common, this implies that $cd$ is vertex-disjoint from at least one of $ab$ and $a'b'$; by symmetry, we may assume that $cd$ is vertex-disjoint from $ab$. This proves our claim. 

So far, we have found vertex-disjoint edges $ab,cd \in E$ such that $ab$ and $cd$ do not belong to the same flat branch of $H$. Since $H$ is a subdivision of a 3-connected graph, we know that $H$ contains a cycle $Z$ such that $ab,cd \in E(Z)$; since $H$ is chordless, the cycle $Z$ of $H$ is induced. By symmetry, we may assume that the vertices $a,b,c,d$ appear in this order in $Z$. Since $H$ is a subdivision of a 3-connected graph, Lemma~\ref{lemma45} implies that $H$ is cyclically 3-connected. Let $P_{ad}$ be the subpath of $Z$ such that the endpoints of $P_{ad}$ are $a$ and $d$, and $b,c \notin V(P_{ad})$, and let $P_{bc}$ be defined in an analogous fashion. By Lemma~\ref{lemma47}, there exists a path $R$ of $H$ such that one endpoint of $R$ belongs to $P_{ad}$, the other endpoint of $R$ belongs to $P_{bc}$, no internal vertex of $R$ belongs to $Z$, and $R$ is not between $a$ and $b$ or between $c$ and $d$. Since $H$ is chordless, the path $R$ is of length at least two. Clearly, $Z \cup R$ is a theta, and since $H$ is chordless, this theta is an induced subgraph of $H$. Further, since $H$ is a subdivision of a 3-connected graph, we know that $H$ is not a theta, and so Lemma~\ref{lemma43} implies that there exists a subgraph $K$ of $H$ such that $K$ is a subdivision of $K_4$, and the theta $Z \cup R$ is a subgraph of $K$; since $H$ is chordless, we know that $K$ is in fact a chordless subdivision of $K_4$ and an induced subgraph of $H$, and that the theta $Z \cup R$ is an induced subgraph of $K$. Since the edges $ab$ and $cd$ do not belong to the same flat branch of the theta $Z \cup R$, we know that $ab$ and $cd$ do not belong to the same flat branch of $K$. This completes the argument.
\end{proof}

\begin{lemma} \label{cyclic3-conn} Let $G$ be an \{ISK4,wheel,diamond\}-free trigraph, let $K$ be an induced subtrigraph of $G$ such that $K$ is a line trigraph of a cyclically 3-connected, chordless graph $H$ of maximum degree at most three, and let $P$ be an inclusion-wise minimal connected induced subtrigraph of $G \smallsetminus V(K)$ such that $P$ is neither of type triangle nor of type branch with respect to $K$. Then $P$ is an augmenting path for $K$ in $G$.
\end{lemma}
\begin{proof}
We first observe that $P$ is a non-null trigraph. Indeed, since $H$ is cyclically 3-connected, $H$ contains a vertex of degree at least three (in fact, since the maximum degree of $H$ is at most three, we know that $H$ contains a vertex of degree exactly three). Since $K$ is a line trigraph of $H$, it follows that $K$ contains a strong triangle. Clearly, the (empty) attachment of the null trigraph over $K$ in $G$ is included in this strong triangle of $K$, and so since $P$ is not of type triangle, we see that $P$ is non-null.

If $H$ is a theta, then $K$ is prism, and the result follows immediately from Lemma~\ref{prism-conn}. So assume that $H$ is not a theta; then by Lemma~\ref{lemma45}, $H$ is a subdivision of a 3-connected graph. Since $H$ is chordless, we know that $H$ is in fact a chordless subdivision of a 3-connected graph. Now, since $P$ is neither of type triangle nor of type branch with respect to $K$, we know that some two edges of $H$ (equivalently: vertices of $K$) in the attachment of $P$ over $K$ are vertex-disjoint. Further, since $P$ is not of type branch with respect to $K$, we know that the edges of $H$ (equivalently: vertices of $K$) in the attachment of $P$ over $K$ do not all belong to the same flat branch of $H$. Let $E$ be the attachment of $P$ over $K$. Proposition~\ref{prop-dist-flat-branch} now implies that $H$ contains an induced subgraph $H'$ such that $H'$ is a chordless subdivision of $K_4$, and such that there exist vertex-disjoint edges $ab,cd \in E \cap E(H')$ that do not belong to the same flat branch of $H'$. Let $K' = K[E(H')]$. Since the vertices $ab$ and $cd$ of $K$ belong to the attachment of $P$ over $K'$, we easily deduce that $P$ is a minimal connected induced subtrigraph of $G \smallsetminus V(K')$ such that $P$ is neither of type triangle nor of type branch with respect to $K'$. By Lemma~\ref{K4-conn} then, we know that $P$ is an augmenting path for $K'$.

If $K = K'$, then $P$ is an augmenting path for $K$, and we are done. So assume that $V(K') \subsetneqq V(K)$. Let $A$ be the attachment of $P$ over $K'$. Since $P$ is an augmenting path for $K'$, we know that $|A| = 4$, and furthermore, there exists a cycle $Z'$ of $H'$ (since $H'$ is chordless, the cycle $Z'$ is induced) such that $A \subseteq E(Z')$. Now, we claim that $A$ is the attachment of $P$ over $K$. Suppose otherwise. Fix $xy \in E(H) \smallsetminus A$ such that $xy$ is in the attachment of $P$ over $K$. We now apply Lemma~\ref{lemma48} to the graph $H$, the cycle $Z'$, and the edge $xy$, and we deduce that $H$ contains a subdivision $H''$ of $K_4$ that contains $Z'$ and $xy$; since $H$ is chordless, so if $H''$. Set $K'' = K[E(H'')] = G[E(H'')]$. Since the attachment of $P$ over $K''$ contains at least five vertices (because it includes $A \cup \{xy\}$), we know that $P$ is not of type triangle with respect to $K$. Further, because $A \subseteq E(Z')$, and $xy \notin E(Z')$, we deduce that $A \cup \{xy\}$ is not included in a flat branch of $K[E(H'')]$, and so $P$ is not of type branch. We now deduce that $P$ is a minimal connected induced subgraph of $G \smallsetminus V(K'')$ such that $P$ is neither of type branch nor of type triangle with respect to $K''$ (the minimality of $P$ with respect to $K''$ follows from the minimality of $P$ with respect to $K$). Thus, by Lemma~\ref{K4-conn}, $P$ is an augmenting path for $K''$. But this is impossible because the attachment of $P$ over $K''$ contains at least five vertices. Thus, the attachment of $P$ over $K$ is precisely $A$.

So far, we have shown that $P$ is an augmenting path for $K'$, and that the attachment of $P$ over $K$ is the same as the attachment of $P$ over $K'$. In order to show that $P$ is augmenting path for $K$, it now only remains to show that neither of the two strongly adjacent pairs of $K$ that belong to the attachment of $P$ over $K$ belongs to any triangle of $K$. But this follows immediately from the fact that $G$ is diamond-free.
\end{proof}

\begin{proposition} \label{max-cyclic3} Let $G$ be an \{ISK4,wheel,diamond\}-free trigraph, let $K$ be an inclusion-wise maximal induced subtrigraph of $G$ such that $K$ is a line trigraph of a cyclically 3-connected, chordless graph $H$ of maximum degree at most three. Then every component of $G \smallsetminus V(K)$ is either of type branch or of type triangle with respect to $K$.
\end{proposition}
\begin{proof}
We prove a slightly stronger statement: every connected induced subtrigraph of $G \smallsetminus V(K)$ is either of type branch or of type triangle with respect to $K$. Suppose otherwise, and let $P$ be a minimal induced subtrigraph of $G \smallsetminus V(K)$ that is neither of type branch nor of type triangle with respect to $K$. Then by Lemma~\ref{cyclic3-conn}, $P$ is an augmenting path for $K$ in $G$. Now we show that $G[V(K) \cup V(P)]$ is a line trigraph of a cyclically 3-connected, chordless graph $H$ of maximum degree at most three (this will contradict the maximality of $K$). Since $K$ is a line trigraph, we know that every triangle of $K$ is strong, and since $P$ is an augmenting path for $K$, we also know that the two triangles formed by the endpoints of $P$ and their attachments over $K$ are strong. It follows that all triangles of $G[V(K) \cup V(P)]$ are strong, and we need only show that the full realization of $G[V(K) \cup V(P)]$ is the line graph of a cyclically 3-connected, chordless graph $H$ of maximum degree at most three. Thus, we may assume that $G[V(K) \cup V(P)]$ contains no semi-adjacent pairs, that is, that $G[V(K) \cup V(P)]$ is a graph. Let $p$ and $p'$ be the endpoints of $P$.

Let $\{ab,bc\} \subseteq E(H) = V(K)$ be the attachment of $p$ over $K$, and let
$\{a'b',b'c'\} \subseteq E(H) = V(K)$ be the attachment of $p'$ over $K$. Since $p$ is of type branch with respect to $K$ (because $P$ is an augmenting path for $K$), we know that the vertices $ab$ and $bc$ of $K$ belong to the same flat branch of $K$, and consequently, that the vertex $b$ (of $H$) is of degree two in $H$.
Similarly, the vertex $b'$ (of $H$) is of degree two in $H$. Let $H'$ be the graph obtained by adding to $H$ a path $R$ between $b$ and $b'$ of length one more than the length of $P$. It is then easy to see that $G[V(K) \cup V(P)]$ is the line graph of $H'$. The fact that $H'$ is cyclically 3-connected follows from
Lemma~\ref{lemma46}, and the fact that $H'$ is chordless and of maximum degree at most three follows from the fact that its line graph is \{ISK4,wheel\}-free. Thus, $G[V(K) \cup V(P)]$ contradicts the maximality of $K$. This completes the argument.
\end{proof}

\begin{lemma} \label{cyclic3-decomp} Let $G$ be an \{ISK4,wheel,diamond\}-free trigraph such that some induced subtrigraph of $G$ is a line trigraph of a cyclically 3-connected, chordless graph of maximum degree at most three. Then either $G$ is a line trigraph of a cyclically 3-connected, chordless graph of maximum degree at most three, or $G$ admits a clique-cutset or a stable 2-cutset.
\end{lemma}
\begin{proof}
Let $K$ be a maximal induced subgraph of $G$ that is a line trigraph of a cyclically 3-connected, chordless graph of maximum degree at most three. If $G = K$, then we are done. So assume that $V(K) \subsetneqq V(G)$. We may also assume that $G$ does not admit a clique-cutset, for otherwise we are done. (In particular then, $G$ is connected, and $G$ does not contain a cut-vertex.) Note that this implies that no component of $G \smallsetminus V(K)$ is of type triangle with respect to $K$, for otherwise, the attachment of such a component over $K$ would be a clique-cutset of $G$. Proposition~\ref{max-cyclic3} now guarantees that all components of $G \smallsetminus V(K)$ are of type branch with respect to $K$. Since $V(K) \subsetneqq V(G)$, and since every component of $G \smallsetminus V(K)$ is of type branch with respect to $K$, we know that there exists some flat branch $B$ of $K$ such that the attachment of some component of $G \smallsetminus V(K)$ is included in $V(B)$. Let $b$ and $b'$ be the endpoints of $B$. Then $\{b,b'\}$ is a cutset of $G$. If $bb'$ were a strongly adjacent pair, then $\{b,b'\}$ would be a clique-cutset of $G$, contrary to the fact that $G$ admits no clique-cutset. So $bb'$ is an anti-adjacent pair, and it follows that $\{b,b'\}$ is a stable 2-cutset.
\end{proof}

\subsection{Decomposing \{ISK4,wheel\}-free graphs}

We remind the reader that a trigraph is {\em series-parallel} if its full realization is series-parallel. We also remind the reader that a {\em complete bipartite trigraph} is a trigraph whose vertex set can be partitioned into two strongly stable sets, strongly complete to each other; if these two strongly stable sets are each of size at least three, then the complete bipartite trigraph is said to be {\em thick}. A trigraph is a {\em strong $K_{3,3}$} if its full realization is a $K_{3,3}$ and it contains no semi-adjacent pairs. (Clearly, a strong $K_{3,3}$ is a thick complete bipartite trigraph.)

\begin{proposition} \label{prism-realization} Let $G$ be an \{ISK4,wheel\}-free trigraph whose full realization contains a prism as an induced subgraph. Then $G$ contains a prism as an induced subtrigraph.
\end{proposition}
\begin{proof}
Let $H$ be an induced subtrigraph of $G$ such that the full realization of $H$ is a prism. Then the two triangles of $H$ are strong, for otherwise, $H$ would be an ISK4-trigraph, contrary to the fact that $G$ is ISK4-free. It follows that the trigraph $H$ is a prism.
\end{proof}

\begin{proposition} \label{K33-strong} Let $G$ be an ISK4-free $K_{3,3}$-trigraph. Then $G$ contains no semi-adjacent pairs (i.e., $G$ is a strong $K_{3,3}$).
\end{proposition}
\begin{proof}
Since $G$ is a $K_{3,3}$-trigraph, we know that $V(G)$ can be partitioned into two stable sets of size three, say $\{a_1,a_2,a_3\}$ and $\{b_1,b_2,b_3\}$, that are complete to each other. Note that $\{a_1,a_2,a_3\}$ and $\{b_1,b_2,b_3\}$ are strongly complete to each other, for otherwise, $G$ would be an ISK4-trigraph, contrary to the fact that $G$ is ISK4-free. Further, $\{a_1,a_2,a_3\}$ is a strongly stable set, for otherwise, $G \smallsetminus b_3$ would be an ISK4-trigraph; similarly, $\{b_1,b_2,b_3\}$ is a strongly stable set. Thus, $G$ contains no semi-adjacent pairs, and $G$ is a strong $K_{3,3}$.
\end{proof}

The following is Lemma 2.2 from~\cite{MR2927414}.

\begin{lemma}\label{lemma22}\cite{MR2927414} Let $G$ be an ISK4-free graph. Then either $G$ is a series-parallel graph, or $G$ contains a prism, a wheel, or a $K_{3,3}$ as an induced subgraph.
\end{lemma}

\begin{lemma} \label{K33-or-prism} Let $G$ be an \{ISK4,wheel\}-free trigraph. Then either $G$ is series-parallel, or $G$ contains a prism or a strong $K_{3,3}$ as an induced subtrigraph.
\end{lemma}
\begin{proof}
Let $\widetilde{G}$ be the full realization of $G$. Then $\widetilde{G}$ is an \{ISK4,wheel\}-free graph, and so by Lemma~\ref{lemma22}, we know that either $\widetilde{G}$ is a series-parallel graph, or $\widetilde{G}$ contains a prism or a $K_{3,3}$ as an induced subgraph. In the former case, $G$ is a series-parallel trigraph, and we are done. So assume that $\widetilde{G}$ contains a prism or a $K_{3,3}$ as an induced subgraph. Then by Propositions~\ref{prism-realization} and~\ref{K33-strong}, $G$ contains a prism or a strong $K_{3,3}$ as an induced subtrigraph.
\end{proof}

\begin{proposition} \label{K33-vertex} Let $G$ be an \{ISK4,wheel\}-free trigraph, let $H$ be a maximal thick bipartite induced subtrigraph of $G$, let $(A,B)$ be a bipartition of $H$, and let $v \in V(G) \smallsetminus V(H)$. Then $v$ has at most one neighbor in $A$ and at most one neighbor in $B$.
\end{proposition}
\begin{proof}
Suppose first that $v$ has at least two neighbors in each of $A$ and $B$. Then there exist distinct $a_1,a_2 \in A$ and distinct $b_1,b_2 \in B$ such that $v$ is complete to $\{a_1,a_2,b_1,b_2\}$; but now $G[v,a_1,a_2,b_1,b_2]$ is a wheel-trigraph, contrary to the fact that $G$ is wheel-free.

Exploiting symmetry, we may now assume that $v$ has at most one neighbor in $B$; then $v$ has at least two strong anti-neighbors in $B$, call them $b_1$ and $b_2$. If $v$ has at most one neighbor in $A$, then we are done. So assume that $v$ has at least two neighbors in $A$, call them $a_1$ and $a_2$. Suppose first that $v$ has at least one strong anti-neighbor in $A$, call it $a$. Then $G[v,a_1,a_2,a,b_1,b_2]$ is an ISK4-trigraph, contrary to the fact that $G$ is ISK4-free. Thus, $v$ is complete to $A$. Suppose first that $v$ has a weak neighbor in $A$, and let $A' \subseteq A$ be such that $v$ has a weak neighbor in $A'$, and $|A'| = 3$. But then $G[A' \cup \{v,b_1,b_2\}]$ is a $K_{3,3}$-trigraph that contains a semi-adjacent pair, contrary to Proposition~\ref{K33-strong}. Thus, $v$ is strongly complete to $A$. If $v$ is strongly anti-complete to $B$, then $G[V(H) \cup \{v\}]$ contradicts the maximality of $H$. It follows that $v$ has a unique neighbor in $B$, call it $b$. But then $v$, $b$, $b_1$, and any two vertices of $A$ induce an ISK4-trigraph in $G$, contrary to the fact that $G$ is ISK4-free.
\end{proof}

\begin{proposition} \label{K33-comp} Let $G$ be an \{ISK4,wheel\}-free trigraph, let $H$ be a maximal induced thick bipartite subtrigraph of $G$, let $(A,B)$ be a bipartition of $H$, and let $C$ be a component of $G \smallsetminus V(H)$. Then the attachment of $C$ over $H$ contains at most one vertex of $A$ and at most one vertex of $B$.
\end{proposition}
\begin{proof}
Suppose otherwise, and fix a minimal connected induced subtrigraph $P$ of $C$ such that the attachment of $P$ over $H$ either contains more than one vertex of $A$, or more than one vertex of $B$. By Proposition~\ref{K33-vertex}, $P$ has at least two vertices.

Let us first show that $P$ is a narrow path. Suppose otherwise. By symmetry, we may assume that the attachment of $P$ over $H$ contains at least two vertices of $A$, call them $a_1$ and $a_2$. For each $i \in \{1,2\}$, let $P_i$ be the set of all vertices of $P$ that are adjacent to $a_i$. By construction, $P_1$ and $P_2$ are non-null, and by Proposition~\ref{K33-vertex}, they are disjoint. Using the fact that $P$ is connected, we let $P'$ be a narrow path in $P$ whose one endpoint belongs to $P_1$, and whose other endpoint belongs to $P_2$. Since $P$ is not a narrow path, we know that $P' \neq P$. But now $P'$ contradicts the minimality of $P$. This proves that $P$ is a narrow path. Let $p$ and $p'$ be the endpoints of $P$. (Since $|V(P)| \geq 2$, we know that $p \neq p'$.)

By symmetry, we may assume that the attachment of $P$ over $H$ contains at least two vertices of $A$. By the minimality of $P$, each end-vertex of $P$ has a neighbor in $A$; by Proposition~\ref{K33-vertex} then, each endpoint of $P$ has a unique neighbor in $A$. Let $a$ be the unique neighbor of $p$ in $A$, and let $a'$ be the unique neighbor of $p'$ in $A$. By the minimality of $P$, we know that $a \neq a'$, and that the interior of $P$ is strongly anti-complete to $A$. In particular then, the attachment of $P$ over $A$ contains exactly two vertices.

By the minimality of $P$, we know that the attachment of $P$ over $H$ contains at most two vertices of $B$. Now, suppose that the attachment of $P$ over $H$ contains exactly two vertices of $B$, call them $b$ and $b'$. Then by the minimality of $P$, we know that one of $p$ and $p'$ is adjacent to $b$ (and strongly anti-adjacent to $b'$), the other one is adjacent to $b'$ (and strongly anti-adjacent to $b$), and the interior of $P$ is strongly anti-complete to $B$. By symmetry, we may assume that $p$ is adjacent to $b$ and strongly anti-adjacent to $b'$, and that $p'$ is adjacent to $b'$ and strongly anti-adjacent to $b$. Using the fact that $|B| \geq 3$, we fix some $b'' \in B \smallsetminus \{b,b'\}$. But then $G[V(P) \cup \{a,a',b,b''\}]$ is an ISK4-trigraph, contrary to the fact that $G$ is ISK4-free.

From now on, we assume that the attachment of $P$ over $H$ contains at most one vertex of $B$. Since $|B| \geq 3$, it follows that there exist two distinct vertices in $B$ (call them $b_1,$ and $b_2$) that are strongly anti-complete to $P$. Further, since $\{a,a'\}$ is the attachment of $P$ over $A$, and since $|A| \geq 3$, there exists some $a'' \in A$ that is strongly anti-complete to $P$. But now $G[V(P) \cup \{a,a',a'',b_1,b_2\}]$ is an ISK4-trigraph, contrary to the fact that $G$ is ISK4-free.
\end{proof}

\begin{lemma} \label{K33-clique-cut} Let $G$ be an \{ISK4,wheel\}-free trigraph that contains a strong $K_{3,3}$ as an induced subtrigraph. Then either $G$ is a thick complete bipartite trigraph, or $G$ admits a clique-cutset.
\end{lemma}
\begin{proof}
Let $H$ be a maximal induced subtrigraph of $G$ such that $H$ is a thick bipartite trigraph (such an $H$ exists because $G$ contains a strong $K_{3,3}$ as an induced subtrigraph). If $H = G$, then $G$ is a thick complete bipartite trigraph, and we are done. So assume that $H \neq G$. Let $(A,B)$ be a bipartition of $H$. Fix a component $C$ of $G \smallsetminus V(H)$. Then by Proposition~\ref{K33-comp}, the attachment of $C$ over $H$ contains at most one vertex of $A$, and at most one vertex of $B$. Since $A$ is strongly complete to $B$, it follows that the attachment of $C$ over $H$ is a clique-cutset of $G$.
\end{proof}

\subsection{Proof of the main theorem}

We are finally ready to prove Theorem~\ref{thm-decomp}, our decomposition theorem for \{ISK4,wheel\}-free trigraphs. In fact, we prove a slightly stronger theorem (Theorem~\ref{thm-decomp-strong}), stated below. Clearly, Theorem~\ref{thm-decomp} is an immediate corollary of Theorem~\ref{thm-decomp-strong}.

\begin{theorem} \label{thm-decomp-strong} Let $G$ be an \{ISK4,wheel\}-free trigraph. Then at least one of the following holds:
\begin{itemize}
\item $G$ is a series-parallel trigraph;
\item $G$ is a thick complete bipartite trigraph;
\item $G$ is a line trigraph of a cyclically 3-connected, chordless graph of maximum degree at most three;
\item $G$ admits a clique-cutset;
\item $G$ admits a stable 2-cutset.
\end{itemize}
\end{theorem}
\begin{proof}
We may assume that $G$ is diamond-free, for otherwise, Proposition~\ref{prop-diamond} guarantees that $G$ admits a clique-cutset or a stable 2-cutset, and we are done. Further, by Lemma~\ref{K33-or-prism}, either $G$ is a series-parallel trigraph, or it contains a prism or a strong $K_{3,3}$ as an induced subtrigraph. If $G$ is series-parallel, then we are done. If $G$ contains a strong $K_{3,3}$ as an induced subtrigraph, then by Lemma~\ref{K33-clique-cut}, either $G$ is a thick complete bipartite trigraph, or $G$ admits a clique-cutset, and in either case, we are done. It remains to consider the case when $G$ contains an induced prism. Note that the prism is a line trigraph of a theta, and a theta is a cyclically 3-connected, chordless graph of maximum degree at most three. Now let $K$ be a maximal induced subtrigraph of $G$ such that $K$ is a line trigraph of a cyclically 3-connected, chordless graph of maximum degree at most three. If $K = G$, then we are done. Otherwise, Lemma~\ref{cyclic3-decomp} guarantees that $G$ admits a clique-cutset or a stable 2-cutset. This completes the argument.
\end{proof}

\end{document}